\documentclass{amsart}
\usepackage{amssymb}
\usepackage{amsmath}
\usepackage{amsfonts}

\newtheorem{theorem}{Theorem}
\theoremstyle{plain}

\newtheorem{corollary}{Corollary}

\newtheorem{lemma}{Lemma}

\numberwithin{equation}{section}
\input{tcilatex}

\begin{document}
\title{On some properties of new paranormed sequence space of non-absolute
type}
\author{Vatan Karakaya}
\address{Department of Mathematical Engineering, Yildiz Technical
University, Davutpasa Campus, Esenler, \.{I}stanbul-Turkey}
\email{vkkaya@yildiz.edu.tr or vkkaya@yahoo.com}
\author{Necip \c{S}im\c{s}ek}
\address{Istanbul Commerce University, Department of Mathematics, Uskudar,
Istanbul,Turkey}
\email{necsimsek@yahoo.com}
\author{Harun Polat}
\address{Department of Mathematics,Faculty of Art and Science,Mu\c{s}
Alparslan University, 49100 Mu\c{s}, Turkey}
\email{h.polat@alparslan.edu.tr}

\begin{abstract}
In this work, we introduce some new generalized sequence space related to
the space $\ell (p)$. Furthermore we investigate some topological properties
as the completeness, the isomorphism and also we give some inclusion
relations between this sequence space and some of the other sequence spaces.
In addition, we compute $\alpha -$, $\beta -$ and $\gamma -$duals of this
space, and characterize certain matrix transformations on this sequence
space.
\end{abstract}

\maketitle

\section{\textbf{Introduction}}

In studying the sequence spaces, especially, to obtain new sequence spaces,
in general, the matrix domain $\mu _{A}$ of an infinite matrix $A$ defined
by $\mu _{A}=\{x=(x_{k})\in w:Ax\in \mu \}$ is used. In the most cases, the
new sequence space $\mu _{A}$ generated by a sequence space $\mu $ is the
expansion or the contraction of the original space $\mu $. In some cases,
these spaces could be overlap. Indeed, one can easily see that the inclusion
$\mu _{S}\subset \mu $ strictly holds for $\mu \in \{\ell _{\infty
},c,c_{0}\}$. Similarly one can deduce that the inclusion $\mu \subset \mu
_{\Delta }$ also strictly holds for $\mu \in \{\ell _{\infty },c,c_{0}\};$
where $S$ and $\Delta $ are matrix operators.

Recently, in \cite{MursaleenNoman}, Mursaleen and Noman constructed new
sequence spaces by using matrix domain over a normed space. They also
studied some topological properties and inclusion relations of these spaces.

It is well known that paranormed spaces have more general properties than
the normed spaces. In this work, we generalize the normed sequence spaces
defined by Mursaleen \cite{MursaleenNoman} to the paranormed spaces.
Furthermore we introduce new sequence space over the paranormed space. Next
we investigate behaviors of this sequence space according to topological
properties and inclusion relations. Finally we give certain matrix
transformation on this sequence space and its duals.

In the literature, by using the matrix domain over the paranormed
spaces, many authors have defined new sequence spaces. Some of
them are as the following. For example; Choudhary and Mishra
\cite{ChMis} have defined the sequence space $\ell
\overline{\left( p\right)}$ which the $S-$transform is
in $\ell \left( p\right)$, Basar and Altay$\left( \text{\cite{PBasarAltay1},%
\cite{PBasarAltay2}}\right)$ defined the spaces $\lambda \left( u,v;p\right)
=\left\{ \lambda \left( p\right) \right\}_{G}$ for $\lambda \in \left\{ \ell
_{\infty },c,c_{0}\right\}$ and $\ell \left( u,v;p\right) =\left\{ \ell
\left( p\right) \right\}_{G}$ respectively, and Altay and Basar \cite%
{AltayBasar1} have defined the spaces $r_{\infty }^{t}\left( p\right),
r_{c}^{t}\left( p\right), r_{0}^{t}\left( p\right)$. In \cite{PKarakayaPolat}%
, Karakaya and Polat defined and examined the spaces $e_{0}^{r}\left( \Delta
;p\right), e^{r}\left( \Delta ;p\right), e_{\infty }^{r}\left( \Delta
;p\right)$, and Karakaya, Noman and Polat \cite{PKarakayaNH} have recently
introduced and studied the spaces $\ell _{\infty }\left( \lambda ,p\right), $
$c\left( \lambda ,p\right) $, $c_{0}\left( \lambda, p\right)$; where $R^{t}$
and $E^{r}$ denote the Riesz and the Euler means, respectively, $\Delta $
denotes the band matrix of the difference operators, and $\Lambda$, $G$ are
defined in \cite{MursaleenNoman} and \cite{Makowskysavas}, respectively.

By $w$, we denote the space of all real valued sequences. Any vector
subspace of $w$ is called a sequence space. By the spaces $\ell _{1}$, $cs$
and $bs$, we denote the spaces of all absolutely convergent series,
convergent series and bounded series, respectively.

A linear topological space $X$ over the real field $\mathbb{R}$ is said to
be a paranormed space if there is a subadditivity function $h:X\rightarrow
\mathbb{R}$ such that $h\left( \theta \right) =0$, $h\left( x\right)
=h\left( -x\right) $ and scalar multiplication is continuous, i.e.; $%
\left\vert \alpha _{n}-\alpha \right\vert \rightarrow 0$ and $h\left(
x_{n}-x\right) \rightarrow 0$ imply $h\left( \alpha _{n}x_{n}-\alpha
x\right) \rightarrow 0$ for all $\alpha $ in $\mathbb{R}$ and $x$ in $X$,
where $\theta $ is the zero in the linear space $X$.

Let $\mu ,\nu $ be any two sequence spaces and let $A=\left( a_{nk}\right) $
be any infinite matrix of real number $a_{nk}$, where $n,k\in \mathbb{N}$
with $\mathbb{N}=\left\{ 0,1,2,...\right\} $. Then we say that $A$ defines a
matrix mapping from $\mu $ into $\nu $ by writing$\ A:\mu \rightarrow \nu ,$
if for every sequence $x=\left( x_{k}\right) \in \mu ,$ the sequence $%
Ax=\left( A_{n}\left( x\right) \right) $, the $A-$transform of $x$, is in $%
\nu ,$ where%
\begin{equation}
A_{n}\left( x\right) =\sum\limits_{k}a_{nk}x_{k}\text{ \ \ }\left( n\in
\mathbb{N}\right) \text{.}  \label{1.1}
\end{equation}%
By $\left( \mu ,\nu \right) $, we denote the class of all matrices $A$ such
that $A:\mu \rightarrow \nu $. Thus, $A\in \left( \mu ,\nu \right) $ if and
only if \ the series on the right hand side of $\left( 1.1\right) $
converges for each $n\in \mathbb{N}$ and every $x\in \mu ,$ and we have $%
Ax\in \nu $ for all $x\in \mu $. A sequence $x$ is said to be $A-$summable
to $a$ if $Ax$ converges to $a$ which is called as the $A-$limit of $x$.

Assume here and after that$\ \left( p_{k}\right) $, $\left( q_{k}\right) $
are bounded sequences of strictly positive real numbers with $\sup p_{k}=H$
and $M=\max \left( 1,H\right) $, also let $\grave{p}_{k}=\frac{p_{k}}{p_{k}-1%
}$ for $1<p_{k}<\infty $ and for all $k\in \mathbb{N}$ . The linear space $%
\ell (p)$ was defined by Maddox \cite{Maddox1967} as follows.%
\begin{equation*}
\ell (p)=\left\{ x=\left( x_{n}\right) \in w:\dsum\limits_{n=0}^{\infty
}\left\vert x_{n}\right\vert ^{p_{n}}<\infty \right\}
\end{equation*}%
which are the complete space paranormed by%
\begin{equation*}
h\left( x\right) =\left( \dsum\limits_{n=0}^{\infty }\left\vert
x_{n}\right\vert ^{p_{n}}\right) ^{\frac{1}{M}}.\text{ }
\end{equation*}%
Throughout this work, by $\digamma $ and $N_{k}$ respectively, we shall
denote the collection of all subsets of $\mathbb{N}$ and the set of all $%
n\in \mathbb{N}$ such that $n\geq k$ and $e=\left( 1,1,1,...\right) .$

\section{\textbf{The sequence space} $\ell \left(\protect\lambda, p\right)$}

In this section, we define the sequence spaces $\ell \left( \lambda
,p\right) $ and prove that this sequence space according to its paranorm are
complete paranormed linear spaces. In \cite{MursaleenNoman}, Mursaleen and
Noman defined the matrix $\Lambda =\left( \lambda _{nk}\right)
_{n,k=0}^{\infty }$ by%
\begin{equation}
\lambda _{nk}=\QDATOPD\{ . {\frac{\lambda _{k}-\lambda _{k-1}}{\lambda _{n}}
;\text{ \ \ }\left( 0\leq k\leq n\right)}{0 ;\text{ \ \ \ \ \ }\left(
k>n\right) }  \label{2.1}
\end{equation}%
where $\lambda =\left( \lambda _{k}\right) _{k=0}^{\infty }$ be a strictly
increasing sequence of positive reals tending to $\infty $, that is, $%
0<\lambda _{0}<\lambda _{1}<...$ and $\lambda _{k}\rightarrow \infty $ as $%
k\rightarrow \infty $. Now, by using $(2.1)$ we define new sequence space as
follows:%
\begin{equation*}
\ell \left( \lambda ,p\right) =\left\{ x=\left( x_{k}\right) \in
w:\sum\limits_{n=0}^{\infty }\left\vert \frac{1}{\lambda _{n}}%
\tsum\limits_{k=0}^{n}\left( \lambda _{k}-\lambda _{k-1}\right)
x_{k}\right\vert ^{p_{n}}<\infty \right\}.
\end{equation*}%
For any $x=(x_{n})\in w$, we define the sequence $y=(y_{n})$, which will
frequently be used, as the $\Lambda $-transform of $x$, i.e., $y=\Lambda (x)$
and hence
\begin{equation}
y_{n}=\sum_{k=0}^{n}\left( \frac{\lambda _{k}-\lambda _{k-1}}{\lambda _{n}}%
\right) x_{k}~~~~(n\in N).  \label{2.2}
\end{equation}%
We now may begin with the following theorem.

\begin{theorem}
The sequence space $\ell \left( \lambda ,p\right) $ is the complete linear
metric space with respect to paranorm defined by
\end{theorem}

\begin{center}
$h\left( x\right) =\left( \sum\limits_{n=0}^{\infty }\left\vert \frac{1}{%
\lambda _{n}}\tsum\limits_{k=0}^{n}\left( \lambda _{k}-\lambda _{k-1}\right)
x_{k}\right\vert ^{p_{n}}\right) ^{\frac{1}{M}}$.
\end{center}

\begin{proof}
The linearity of $\ell \left( \lambda ,p\right) $ with respect to the
coordinatewise addition and scalar multiplication follows from the following
inequalities which are satisfied for $x,t\in \ell \left( \lambda ,p\right) $
(see; \cite{Maddoxelmt1988}).%
\begin{equation}
\left( \sum\limits_{n=0}^{\infty }\left\vert \frac{1}{\lambda _{n}}%
\tsum\limits_{k=0}^{n}\left( \lambda _{k}-\lambda _{k-1}\right) \left(
x_{k}+t_{k}\text{ }\right) \right\vert ^{p_{n}}\right) ^{\tfrac{1}{M}}\leq
\left( \sum\limits_{n=0}^{\infty }\left\vert \frac{1}{\lambda _{n}}%
\tsum\limits_{k=0}^{n}\left( \lambda _{k}-\lambda _{k-1}\right)
x_{k}\right\vert ^{p_{n}}\right) ^{\tfrac{1}{M}}  \label{2.3}
\end{equation}%
\begin{equation*}
\text{ \ \ \ \ \ \ \ \ \ \ \ \ \ \ \ \ \ \ \ \ \ \ \ \ \ \ \ \ \ \ \ \ \ \ \
\ \ \ \ \ \ \ \ \ \ \ \ \ \ \ \ \ \ \ \ \ \ \ \ \ \ \ }+\left(
\sum\limits_{n=0}^{\infty }\left\vert \frac{1}{\lambda _{n}}%
\tsum\limits_{k=0}^{n}\left( \lambda _{k}-\lambda _{k-1}\right) t_{k}\text{ }%
\right\vert ^{p_{n}}\right) ^{\tfrac{1}{M}}
\end{equation*}%
and for any $\alpha \in \mathbb{R}$ (see;\cite{I.J.Maddox1968})%
\begin{equation}
\left\vert \alpha \right\vert ^{p_{k}}\leq \max \left\{ 1,\left\vert \alpha
\right\vert ^{M}\right\} .  \label{2.4}
\end{equation}%
It is clear that $h\left( \theta \right) =0$, $h\left( x\right) =h\left(
-x\right) $ for all $x\in \ell \left( \lambda ,p\right) $. Again the
inequalities (2.3) and (2.4) yield the subadditivity of $h$ and hence $%
h\left( \alpha x\right) \leq \max \left\{ 1,\left\vert \alpha \right\vert
^{M}\right\} h\left( x\right) $. Let $\left\{ x^{m}\right\} $ be any
sequence of points $x^{m}\in \ell \left( \lambda ,p\right) $ such that $%
h\left( x^{m}-x\right) \rightarrow 0$ and $\left( \alpha _{m}\right) $ also
be any sequence of scalars such that $\alpha _{m}\rightarrow \alpha $. Then,
since the inequality%
\begin{equation*}
h\left( x^{m}\right) \leq h\left( x\right) +h\left( x^{m}-x\right)
\end{equation*}%
holds by subadditivity of $h$, we can write that $\left\{ h\left(
x^{m}\right) \right\} $ is bounded and we thus have%
\begin{eqnarray*}
h\left( \alpha _{m}x^{m}-\alpha x\right) &=&\left( \sum\limits_{n=0}^{\infty
}\left\vert \frac{1}{\lambda _{n}}\tsum\limits_{k=0}^{n}\left( \lambda
_{k}-\lambda _{k-1}\right) \left( \alpha _{m}x_{k}^{m}-\alpha x_{k}\right)
\right\vert ^{p_{n}}\right) ^{\tfrac{1}{M}} \\
&\leq &\left\vert \alpha _{m}\rightarrow \alpha \right\vert ^{\frac{1}{M}%
}h\left( x^{m}\right) +\left\vert \alpha \right\vert ^{\frac{1}{M}}h\left(
x^{m}-x\right)
\end{eqnarray*}%
which tends to zero as $n\rightarrow \infty $. Therefore, the scalar
multiplication is continuous. Hence $h$ is a paranorm on the space $\ell
\left( \lambda ,p\right) $. It remains to prove the completeness of the
space $\ell \left( \lambda ,p\right) $. Let $\left\{ x^{j}\right\} $ be any
Cauchy sequence in the space $\ell \left( \lambda ,p\right) $, where $%
x^{j}=\left\{ x_{0}^{\left( j\right) },x_{1}^{\left( j\right)
},x_{2}^{\left( j\right) },...\right\} $. Then, for a given $\varepsilon >0,$
there exists a positive integer $m_{0}\left( \varepsilon \right) $ such that
$h\left( x^{j}-x^{i}\right) <\frac{\varepsilon }{2}$ for all $i$,$%
j>m_{0}\left( \varepsilon \right) $. Using definition of $h,$ we obtain for
each fixed $n\in \mathbb{N}$ that%
\begin{equation}
\left\vert \Lambda _{n}\left( x^{j}\right) -\Lambda _{n}\left( x^{i}\right)
\right\vert \leq \left( \sum\limits_{n=0}^{\infty }\left\vert \Lambda
_{n}\left( x^{j}\right) -\Lambda _{n}\left( x^{i}\right) \right\vert
^{p_{n}}\right) ^{\tfrac{1}{M}}<\frac{\varepsilon }{2}  \label{2.5}
\end{equation}%
for every $i$,$j>m_{0}\left( \varepsilon \right) $ which leads us to the
fact that $\left\{ \Lambda _{n}\left( x^{0}\right) ,\Lambda _{n}\left(
x^{1}\right) ,\Lambda _{n}\left( x^{2}\right) ,...\right\} $ is a Cauchy
sequence of real numbers for every fixed $n\in \mathbb{N}$. Since $\mathbb{R}
$ is complete, it converges, say $\Lambda _{n}\left( x^{i}\right) -\Lambda
_{n}\left( x\right) $ as $i\rightarrow \infty $. Using these infinitely many
limits, we may write the sequence $\left\{ \Lambda _{0}\left( x\right)
,\Lambda _{1}\left( x\right) ,\Lambda _{2}\left( x\right) ,...\right\} $.
From $(2.5)$ as $i\rightarrow \infty $, we have%
\begin{equation*}
\left\vert \Lambda _{n}\left( x^{j}\right) -\Lambda _{n}\left( x\right)
\right\vert <\frac{\varepsilon }{2},\left( j\geq m_{0}\left( \varepsilon
\right) \right)
\end{equation*}%
for every fixed $n\in \mathbb{N}$. Since $x^{j}=\left( x_{k}^{\left(
j\right) }\right) \in \ell \left( \lambda ,p\right) $ for each $j\in \mathbb{%
N}$, there exists $m_{0}\left( \varepsilon \right) \in \mathbb{N}$ such that
$\left( \sum\limits_{n=0}^{\infty }\left\vert \Lambda _{n}\left(
x^{j}\right) \right\vert ^{p_{_{n}}}\right) ^{\tfrac{1}{M}}<\frac{%
\varepsilon }{2}$ for every $j\geq m_{0}\left( \varepsilon \right) $ and for
each $n\in \mathbb{N}$. By taking a fixed $j\geq m_{0}\left( \varepsilon
\right) $, we obtain by $(2.5)$ that%
\begin{equation*}
\left( \sum\limits_{n=0}^{\infty }\left\vert \Lambda _{n}\left( x\right)
\right\vert ^{p_{_{n}}}\right) ^{\tfrac{1}{M}}\leq \left(
\sum\limits_{n=0}^{\infty }\left\vert \Lambda _{n}\left( x^{j}\right)
-\Lambda _{n}\left( x^{i}\right) \right\vert ^{p_{_{n}}}\right) ^{\tfrac{1}{M%
}}+\left( \sum\limits_{n=0}^{\infty }\left\vert \Lambda _{n}\left(
x^{j}\right) \right\vert ^{p_{_{n}}}\right) ^{\tfrac{1}{M}}<\infty .
\end{equation*}%
Hence, we get $x\in \ell \left( \lambda ,p\right) $. So, the space $\ell
\left( \lambda ,p\right) $ is complete.
\end{proof}

\begin{theorem}
The sequence space $\ell \left( \lambda ,p\right) $ of non-absolute type is
linearly isomorphic to the space $\ell \left( p\right) $; where $0<p_{k}\leq
H<\infty $.
\end{theorem}

\begin{proof}
To prove the theorem, we should show the existence of linear bijection
between the spaces $\ell \left( \lambda ,p\right) $ and $\ell \left(
p\right) $. With the notation of $(2.2)$, we define transformation $T$ from $%
\ell \left( \lambda ,p\right) $ to $\ell \left( p\right) $ by $x\rightarrow
y=Tx$. The linearity of $T$ is trivial. Furthermore, it is obvious that $%
x=\theta $ whenever $Tx=\theta $ and hence $T$ is injective.

Let $y\in \ell \left( p\right) $ and define the sequence $x=\left\{
x_{n}\right\} $%
\begin{equation*}
x_{n}\left( \lambda \right) =\tsum\limits_{k=n-1}^{n}\left( \left( -\right)
^{n-k}\frac{\lambda _{k}}{\lambda _{n}-\lambda _{n-1}}\right) y_{k}\qquad
\left( n,k\in \mathbb{N}\right) .
\end{equation*}%
Then, we have%
\begin{equation*}
h_{\ell \left( \lambda ,p\right) }\left( x\right) =\left(
\sum\limits_{n=0}^{\infty }\left\vert \frac{1}{\lambda _{n}}%
\tsum\limits_{k=0}^{n}\left( \lambda _{k}-\lambda _{k-1}\right)
x_{k}\right\vert ^{p_{n}}\right) ^{\tfrac{1}{M}}=\left(
\sum\limits_{n=0}^{\infty }\left\vert y_{n}\right\vert ^{p_{n}}\right) ^{%
\tfrac{1}{M}}=h_{\ell \left( p\right) }\left( y\right) .
\end{equation*}%
Thus, we have that $x\in \ell \left( \lambda ,p\right) $ and consequently $T$
is surjective. Hence, $T$ is a linear bijection and this says us that the
spaces $\ell \left( \lambda ,p\right) $ and $\ell \left( p\right) $ linearly
isomorphic. This completes the proof.
\end{proof}

\section{\textbf{Some inclusion relations}}

In this section, we give some inclusion relations concerning the
space $\ell \left( \lambda ,p\right) $. Before giving the theorems
about the section, we give a Lemma given in \cite{MursaleenNoman}.

\begin{lemma}
For any sequence $x=\left( x_{k}\right) \in w,$ the equalities%
\begin{equation}
S_{n}\left( x\right) =x_{n}-\Lambda _{n}\left( x\right)  \label{3.1}
\end{equation}%
and%
\begin{equation*}
S_{n}\left( x\right) =\frac{\lambda _{n-1}}{\lambda _{n}-\lambda _{n-1}}%
\left[ \Lambda _{n}\left( x\right) -\Lambda _{n-1}\left( x\right) \right]
\end{equation*}%
hold, where the sequence $S\left( x\right) =\left\{ S_{n}\left( x\right)
\right\} $ is defined by%
\begin{equation*}
S_{0}\left( x\right) =0\text{ and }S_{n}\left( x\right) =\frac{1}{\lambda
_{n}}\tsum\limits_{k=1}^{n}\lambda _{k-1}\left( x_{k}-x_{k-1}\right) \text{
\ \ }\left( n\geq 1\right).
\end{equation*}
\end{lemma}

\begin{theorem}
The inclusion $\ell \left( \lambda ,p\right) \subset $ $c_{0}\left( \lambda
,p\right) $ strictly holds.
\end{theorem}

\begin{proof}
Let $x\in $ $\ell \left( \lambda ,p\right) .$ It can be written $\Lambda
x\in \ell \left( p\right) .$ By the definition of the space $\ell \left(
p\right) ,$ $\Lambda _{n}x\rightarrow \infty $ as $n\rightarrow \infty ,$ we
obtain $\Lambda x\in c_{0}.$ Hence we get $x\in c_{0}\left( \lambda
,p\right) .$

To show strict of the inclusion, by taking $x_{n}^{^{\prime }}=\frac{1}{n+1}%
, $ $p_{k}=1+\frac{1}{n+1},$ we consider the sequence $|x|^{p}={%
(|x_{k}|^{p_{k}})}_{k=0}^{\infty }.$ Then it is easy to see that $\Lambda
\left( |x|^{p}\right) \in c_{0}\left( p\right) .$ Since $c_{0}\left(
p\right) \subset $ $c_{0}\left( \lambda ,p\right) ,$ $x\in c_{0}\left(
\lambda ,p\right) $ $\left( \text{see;\cite{PKarakayaNH}}\right) .$ Hence%
\begin{equation*}
\left\vert \Lambda _{n}\left( x\right) \right\vert \geq \frac{1}{\left(
n+1\right) ^{\frac{1}{1+\frac{1}{n+1}}}} .
\end{equation*}%
This shows that $\Lambda x\notin \ell \left( p\right) $ and hence $x\notin
\ell \left( \lambda ,p\right) .$ Thus the sequence $x$ is in $c_{0}\left(
\lambda ,p\right) $ but not in $\ell \left( \lambda ,p\right) .$
\end{proof}

\begin{theorem}
The inclusion $\ell \left( \lambda ,p\right) \subset \ell \left( p\right) $
if and only if $S\left( x\right) \in $ $\ell \left( p\right) $ for every
sequence $x\in \ell \left( \lambda ,p\right) ;$ where $1\leq p_{k}\leq H.$
\end{theorem}

\begin{proof}
We suppose that $\ell \left( \lambda ,p\right) \subset \ell \left( p\right) $
holds and take any $x\in \ell \left( \lambda ,p\right) .$ Then $x\in \ell
\left( p\right) $ by hypothesis. Thus we obtain from $\left( 3.1\right) $
that%
\begin{equation*}
\left[ h\left( S\left( x\right) \right) \right] _{\ell \left( p\right) }\leq %
\left[ h\left( x\right) \right] _{\ell \left( p\right) }+\left[ h\left(
\Lambda x\right) \right] _{\ell \left( p\right) }=\left[ h\left( x\right) %
\right] _{\ell \left( p\right) }+\left[ h\left( x\right) \right] _{\ell
\left( \lambda ,p\right) }
\end{equation*}%
which yields that $S\left( x\right) \in \ell \left( p\right) .$

Conversely, let $x\in \ell \left( \lambda ,p\right) $ be given. Then we have
by the hypothesis that $S\left( x\right) \in \ell \left( p\right) .$ Again
by using $\left( 3.1\right) $%
\begin{equation*}
\left[ h\left( x\right) \right] _{\ell \left( p\right) }\leq \left[ h\left(
S\left( x\right) \right) \right] _{\ell \left( p\right) }+\left[ h\left(
\Lambda x\right) \right] _{\ell \left( p\right) }=\left[ h\left( S\left(
x\right) \right) \right] _{\ell \left( p\right) }+\left[ h\left( x\right) %
\right] _{\ell \left( \lambda ,p\right) }
\end{equation*}%
which shows that $x\in \ell \left( p\right) .$ Hence the inclusion $\ell
\left( \lambda ,p\right) \subset \ell \left( p\right) $ holds. This
completes the proof.
\end{proof}

\begin{theorem}
$\left( i\right) $ If $p_{n}>1$ for all $n\in N$, then the inclusion $\ell
_{p}^{\lambda }\subset \ell \left( \lambda ,p\right) $ holds.

$\left( ii\right) $ If $p_{n}<1$ for all $n\in N$, then the inclusion $\ell
\left( \lambda ,p\right) \subset \ell _{p}^{\lambda }$ holds.
\end{theorem}

\begin{proof}
$\left( i\right) $ Let $x\in \ell _{p}^{\lambda }.$ It is clear that $%
\Lambda \left( x\right) \in \ell _{p}.$ One can find $m\in N$ such that $%
\left\vert \Lambda _{n}\left( x\right) \right\vert <1$ for all $n\geq m.$
Under the condition $\left( i\right) ,$ we have $\left\vert \Lambda
_{n}\left( x\right) \right\vert ^{p_{n}}<\left\vert \Lambda _{n}\left(
x\right) \right\vert $ for all $n\geq m.$ Hence we get $x\in \ell \left(
\lambda ,p\right) .$

$\left( ii\right) $ We suppose that $x\in \ell \left( \lambda ,p\right) .$
Then $\ \Lambda \left( x\right) \in \ell \left( p\right) $ and there exists $%
m\in N$ such that $\left\vert \Lambda _{n}\left( x\right) \right\vert
^{p_{n}}<1$ for all $n\geq m.$ To obtain the result, we consider the
following inequality;%
\begin{equation*}
\left\vert \Lambda _{n}\left( x\right) \right\vert =\left( \left\vert
\Lambda _{n}\left( x\right) \right\vert ^{p_{n}}\right) ^{\frac{1}{p_{n}}%
}<\left\vert \Lambda _{n}\left( x\right) \right\vert ^{p_{n}}
\end{equation*}%
for all $n\geq m.$ So, we get $x\in \ell _{p}^{\lambda }.$
\end{proof}

\section{\textbf{Some matrix transformations and duals of the space }$\ell
\left( \protect\lambda ,p\right) $}

In this section, we give the theorems determining the $\alpha -,\beta -$ and
$\gamma -$ duals of the space $\ell \left( \lambda ,p\right) $. In proving
the theorem, we apply the technique used in \cite{PBasarAltay1}. Also we
give some matrix transformations from the space $\ell \left( \lambda
,p\right) $ into paranormed spaces $\ell \left( q\right) $ by using the
matrix given in \cite{MursaleenNoman}.

For the sequence space $\mu $ and $\nu $, the set $S\left( \mu ,\nu \right) $%
\ defined by
\begin{equation*}
S\left( \mu ,\nu \right) =\left\{ a=\left( a_{k}\right) \in w:ax\in \nu
\text{ for all }x\in \mu \right\}
\end{equation*}%
is called the multiplier space of $\mu $ and $\nu $. The $\alpha -,\beta -$
and $\gamma -$duals of a sequence space $\mu ,$ which are respectively
denote by $\mu ^{\alpha }$, $\mu ^{\beta }$ and $\mu ^{\gamma }$ are defined
by
\begin{equation*}
\mu ^{\alpha }=S\left( \mu ,\ell _{1}\right) ,\text{ \ \ \ \ \ \ \ \ \ \ \ \
\ \ \ \ }\mu ^{\beta }=S\left( \mu ,cs\right) ,\text{ \ \ \ \ \ \ \ \ \ \ \
\ }\mu ^{\gamma }=S\left( \mu ,bs\right) .
\end{equation*}%
We may begin with the following theorem which computes the $\alpha $-dual of
the space $\ell \left( \lambda ,p\right) $.

\begin{theorem}
Let $K_{1}=\left\{ k\in \mathbb{N}:p_{k}\leq 1\right\} $ and $\
K_{2}=\left\{ k\in \mathbb{N}:p_{k}>1\right\} $. Define the matrix $%
D^{a}=\left( d_{nk}^{a}\right) $ by%
\begin{equation}
d_{nk}^{a}=\left\{
\begin{array}{cc}
\left( -1\right) ^{n-k}\frac{\lambda _{k}}{\lambda _{n}-\lambda _{n-1}}a_{n}
& ,\left( n-1\leq k\leq n\right) \\
0 & ,\left( 0\leq k\leq n-1\right) \text{ or }\left( k>n\right)%
\end{array}%
\right. .  \label{4.1}
\end{equation}%
Then
\begin{equation*}
\ell _{K_{1}}^{\alpha }\left( \lambda ,p\right) =\left\{ a=\left(
a_{n}\right) \in w:D^{a}\in \left( \ell \left( p\right) ;\ell _{\infty
}\right) \right\}
\end{equation*}%
\begin{equation*}
\ell _{K_{2}}^{\alpha }\left( \lambda ,p\right) =\left\{ a=\left(
a_{n}\right) \in w:D^{a}\in \left( \ell \left( p\right) ;\ell _{1}\right)
\right\} .
\end{equation*}
\end{theorem}

\begin{proof}
We consider the following equality%
\begin{equation}
a_{n}x_{n}=\tsum\limits_{k=n-1}^{n}d_{nk}^{a}y_{k}=\left( D^{a}y\right) _{n}%
\text{ \ }\left( n\in \mathbb{N}\right)  \label{4.2}
\end{equation}%
where $D^{a}=\left( d_{nk}^{a}\right) $ is defined by $\left( 4.1\right) .$

From $\left( 4.2\right) $, it can be obtained \ that $ax=\left(
a_{n}x_{n}\right) \in \ell _{1}$ or $ax=\left( a_{n}x_{n}\right) \in \ell
_{\infty }$ whenever $x\in \ell \left( \lambda ,p\right) $ if and only if $%
D^{a}y\in \ell _{1}$ or $D^{a}y\in \ell _{\infty }$ whenever $y\in \ell
\left( p\right).$ This means $a\in \ell _{K_{1}}^{\alpha }\left( \lambda
,p\right) $ or $a\in \ell _{K_{2}}^{\alpha }\left( \lambda ,p\right) $ if
and only if $D^{a}\in \left( \ell \left( p\right) ;\ell _{1}\right) $ or $%
D^{a}\in \left( \ell \left( p\right) ;\ell _{\infty }\right) .$ Hence this
completes the proof.
\end{proof}

The result of the Theorem above corresponds the Theorem 5.1 $\left(
0,8,12\right) $ given in \cite{Gro-Erd1993}.

As a direct consequence of the Theorem 6, we have the following.

\begin{corollary}
Let $K^{\ast }=\left\{ k\in \mathbb{N}:n-1\leq k\leq n\right\} \cap K$ for $%
K\in \digamma $. Then

$\left( i\right) $ $\ell _{K_{1}}^{\alpha }\left( \lambda ,p\right) =\left\{
a=\left( a_{n}\right) \in w:\sup_{N}\sup_{k\in \mathbb{N}}\left\vert
\tsum\limits_{n\in K^{\ast }}d_{nk}^{a}\right\vert ^{p_{k}}<\infty \right\}
; $
\end{corollary}

$\ \ \ \ \ \left( ii\right) $ $\ell _{K_{2}}^{\alpha }\left( \lambda
,p\right) =\bigcup\limits_{M>1}\left\{ a=\left( a_{n}\right) \in
w:\sup_{K\in \digamma }\sum\limits_{k}\left\vert \tsum\limits_{n\in K^{\ast
}}d_{nk}^{a}M^{-1}\right\vert ^{p_{k}^{!}}<\infty \right\} $

In the following theorem, we characterize $\ $the $\beta -$ and $\gamma -$
duals of the space $\ell \left( \lambda ,p\right)$.

\begin{theorem}
Let $K_{1}=\left\{ k\in \mathbb{N}:p_{k}\leq 1\right\} $, $K_{2}=\left\{
k\in \mathbb{N}:p_{k}>1\right\} ,$ and let $\Delta x_{k}=x_{k}-x_{k+1}$.
Define the sequence $s^{1}=\left( s_{k}^{1}\right) ,$ $s^{2}=\left(
s_{k}^{2}\right) $ and the matrix $B^{a}=\left( b_{nk}^{a}\right) $ by

$s_{k}^{1}=\Delta \left( \frac{a_{k}}{\lambda _{k}-\lambda _{k-1}}\right)
\lambda _{k},$ \ \ \ \ \ \ \ \ \ \ \ \ \ \ \ \ $s_{k}^{2}=\frac{a_{k}\lambda
_{k}}{\lambda _{k}-\lambda _{k-1}}$%
\begin{equation*}
b_{nk}^{a}=\left\{
\begin{array}{cc}
s_{k}^{1} & ,\left( 0\leq k\leq n-1\right) \\
s_{k}^{2} & ,\left( k=n\right) \\
0 & ,\left( k>n\right)%
\end{array}%
\right. .
\end{equation*}%
for all $n,k\in \mathbb{N}.$ Then
\begin{equation}
\ell _{K_{1}}^{\beta }\left( \lambda ,p\right) =\ell _{K_{1}}^{\gamma
}\left( \lambda ,p\right) =\left\{ a=\left( a_{n}\right) \in w:B^{a}\in
\left( \ell \left( p\right) ;\ell _{\infty }\right) \right\} ;  \label{4.3}
\end{equation}%
and
\begin{equation*}
\ell _{K_{2}}^{\beta }\left( \lambda ,p\right) =\ell _{K_{2}}^{\gamma
}\left( \lambda ,p\right) =\left\{ a=\left( a_{n}\right) \in w:B^{a}\in
\left( \ell \left( p\right) ;c\right) \right\} .
\end{equation*}
\end{theorem}

\begin{proof}
Consider the equality%
\begin{equation}
\tsum\limits_{k=0}^{n}a_{k}x_{k}=\tsum%
\limits_{k=0}^{n-1}s_{k}^{1}y_{k}+s_{n}^{2}y_{n}=\left( B^{a}y\right) _{n}
\label{4.4}
\end{equation}%
From $\left( 4.4\right) $, it can be obtained \ that $ax=\left(
a_{n}x_{n}\right) \in cs$ or $bs$ whenever $x=\left( x_{n}\right) \in \ell
\left( \lambda ,p\right) $ if and only if $B^{a}y\in c$ or $\ell _{\infty }$
whenever $y=\left( y_{k}\right) \in \ell \left( p\right) .$ This means that $%
a=\left( a_{n}\right) \in \left\{ \ell _{K_{1}}^{\beta }\left( \lambda
,p\right) \text{ or }\ell _{K_{2}}^{\beta }\left( \lambda ,p\right) \right\}
$ or $a=\left( a_{n}\right) \in \left\{ \ell _{K_{1}}^{\gamma }\left(
\lambda ,p\right) \text{ or }\ell _{K_{2}}^{\gamma }\left( \lambda ,p\right)
\right\} $ if and only if $B^{a}\in \left( \ell \left( p\right) ;c\right) $
or $B^{a}\in \left( \ell \left( p\right) ;\ell _{\infty }\right) .$ Hence
this completes the proof.
\end{proof}

We can write the following corollary from \ the Theorem 7.

\begin{corollary}
Let $\grave{p}_{k}=\frac{p_{k}}{p_{k}-1}$ for $1<p_{k}<\infty $ and for all $%
k\in \mathbb{N}.$ Then

$\left( i\right) $ $\ell _{K_{1}}^{\beta }\left( \lambda ,p\right) =\ell
_{K_{1}}^{\gamma }\left( \lambda ,p\right) =\left\{ a=\left( a_{n}\right)
\in w:s^{1},s^{2}\in \ell _{\infty }\left( p\right) \right\} ;$

$\left( ii\right) $ $\ell _{K_{2}}^{\beta }\left( \lambda ,p\right) =\ell
_{K_{2}}^{\gamma }\left( \lambda ,p\right) =\bigcup\limits_{M>1}\left\{
a=\left( a_{n}\right) \in w:s^{1}M^{-1},s^{2}M^{-1}\in \ell \left(
p^{^{\prime }}\right) \cap \ell _{\infty }\left( p^{^{\prime }}\right)
\right\}$.
\end{corollary}

After this step, we can give our theorems on the characterization of some
matrix classes concerning \ with the sequence space $\ell \left( \lambda
,p\right) .$

Let $x,y\in w$ be connected by the relation $y=\Lambda (x)$. For an infinite
matrix $A=(a_{nk})$, we have by using (4.4) of Theorem 7 that
\begin{equation}
\sum_{k=0}^{m}a_{nk}x_{k}=\sum_{k=0}^{m-1}\tilde{a}_{nk}y_{k}+\frac{\lambda
_{m}}{\lambda _{m}-\lambda _{m-1}}a_{nm}y_{m}~~~~~(m,n\in \mathbb{N})
\label{4.5}
\end{equation}%
where
\begin{equation*}
\tilde{a}_{nk}=\left( \frac{a_{nk}}{\lambda _{k}-\lambda _{k-1}}-\frac{%
a_{n,k+1}}{\lambda _{k+1}-\lambda _{k}}\right) \lambda _{k};~~~~(n,k\in
\mathbb{N}).
\end{equation*}%
The necessary and sufficient conditions characterizing the matrix mapping of
the sequence space $\ell \left( p\right) $ of Maddox have been determined by
Grosse-Erdmann \cite{Gro-Erd1993}. Let $L$ and $M$ be the natural numbers
and define the sets by $K_{1}=\left\{ k\in \mathbb{N}:p_{k}\leq 1\right\} $
and $K_{2}=\left\{ k\in \mathbb{N}:p_{k}>1\right\} $ also let us put $\grave{%
p}_{k}=\frac{p_{k}}{p_{k}-1}$ for $1<p_{k}<\infty $ and for all $k\in
\mathbb{N}.$ Before giving the theorems, let us suppose that $\left(
q_{n}\right) $ is a non-decreasing bounded sequence of positive real numbers
and consider the following conditions:

$\sup_{N}$ sup$_{k\in K_{1}}\left\vert \tsum\limits_{n\in N}\tilde{a}%
_{nk}\right\vert ^{q_{n}}<\infty ,\ \ \ \ \ \ \ \ \ \ \ \ \exists M$ $%
\sup_{N}\tsum\limits_{k\in K_{2}}\left\vert \tsum\limits_{n\in N}\tilde{a}%
_{nk}M^{-1}\right\vert ^{\grave{p}_{k}}<\infty ,$

\ \ \ \ \ \ \ \ \ \ \ \ \ \ \ \ \ \ \ \ $\uparrow \left( 4.6\right) $ \ \ \
\ \ \ \ \ \ \ \ \ \ \ \ \ \ \ \ \ \ \ \ \ \ \ \ \ \ \ \ \ \ $\ \ \ \ \ \ \ \
\ \uparrow \left( 4.7\right) $

$\exists M$ sup$_{k}\tsum\limits_{n}\left\vert \tilde{a}_{k}M^{-\frac{1}{%
p_{k}}}\right\vert ^{q_{n}}<\infty ,\ \ \ \ \ \ \ \ \ \ \ \ \ \
\lim_{n}\left\vert \tilde{a}_{nk}\right\vert ^{q_{n}}=0$ $\left( \forall
k\in \mathbb{N}\right) ,$

\ \ \ \ \ \ \ \ \ \ \ \ \ \ \ \ \ \ \ $\uparrow \left( 4.8\right) $ \ \ \ \
\ \ \ \ \ \ \ \ \ \ \ \ \ \ \ \ \ \ \ \ \ \ \ \ \ \ \ \ \ \ \ \ \ \ $\ \ \ \
\ \uparrow \left( 4.9\right) $

$\forall L,$ sup$_{n}$ $sup_{k\in K_{1}}\left\vert \tilde{a}_{nk}L^{^{\frac{1%
}{q_{n}}}}\right\vert ^{p_{k}}<\infty ,\ \ \ \ \ \ \forall L,\exists M$ $%
sup_{n}\tsum\limits_{k\in K_{2}}\left\vert \tilde{a}_{k}L^{^{\frac{1}{q_{n}}%
}}M^{-1}\right\vert ^{\grave{p}_{k}}<\infty ,$

\ \ \ \ \ \ \ \ \ \ \ \ \ \ \ \ \ \ \ \ $\uparrow 4.10$ \ \ \ \ \ \ \ \ \ \
\ \ \ \ \ \ \ \ \ \ \ \ \ \ \ \ \ \ \ \ \ \ \ \ \ \ \ \ \ \ \ \ \ \ $%
\uparrow \left( 4.11\right) $\

$\sup_{n}$ sup$_{k\in K_{1}}\left\vert \tilde{a}_{nk}\right\vert
^{p_{k}}<\infty ,\ \ \ \ \ \ \ \ \ \ \ \ \ \ \ \ \ \ \ \ \ \exists M$ $%
sup_{n}\tsum\limits_{k\in K_{2}}\left\vert \tilde{a}_{k}M^{-1}\right\vert ^{%
\grave{p}_{k}}<\infty ,$

\ \ \ \ \ \ \ \ \ \ \ \ \ \ \ \ \ \ \ \ \ $\uparrow \left( 4.12\right) $ \ \
\ \ \ \ \ \ \ \ \ \ \ \ \ \ \ \ \ \ \ \ \ \ \ \ \ \ \ \ \ \ \ \ \ \ \ \ $\
\uparrow \left( 4.13\right) $

$\forall L,$ sup$_{n}sup_{_{k\in K_{1}}}\left( \left\vert \tilde{a}_{nk}-%
\tilde{a}_{k}\right\vert L^{^{\frac{1}{q_{n}}}}\right) ^{p_{k}}<\infty ,$ \
\ \ $\lim_{n}\left\vert \tilde{a}_{nk}-\tilde{a}_{k}\right\vert ^{q_{n}}=0,$%
for all $k.$

\ \ \ \ \ \ \ \ \ \ \ \ \ \ \ \ \ \ \ \ $\uparrow \left( 4.14\right) $ \ \ \
\ \ \ \ \ \ \ \ \ \ \ \ \ \ \ \ \ \ \ \ \ \ \ \ \ \ \ \ \ \ \ \ \ \ \ \ \ \
\ \ \ \ $\uparrow \left( 4.15\right) $

$\forall L,\exists M$ sup$_{n}\tsum\limits_{_{k\in K_{2}}}\left( \left\vert
\tilde{a}_{nk}-\tilde{a}_{k}\right\vert L^{^{\frac{1}{q_{n}}}}M^{-1}\right)
^{\grave{p}_{k}},$ $\exists L, $sup$_{n}sup_{_{k\in K_{1}}}\left\vert \tilde{%
a}_{nk}L^{^{-\frac{1}{q_{n}}}}\right\vert ^{p_{k}}<\infty ,$

\ \ \ \ \ \ \ \ \ \ \ \ \ \ \ \ \ \ \ \ $\uparrow \left( 4.16\right) $ \ \ \
\ \ \ \ \ \ \ \ \ \ \ \ \ \ \ \ \ \ \ \ \ \ \ \ \ \ \ \ \ \ \ \ \ \ \ \ \ \
\ \ \ \ \ \ \ $\uparrow \left( 4.17\right) $

$\exists L,$sup$_{n}\tsum\limits_{_{k\in K_{2}}}\left\vert \tilde{a}%
_{nk}L^{^{-\frac{1}{q_{n}}}}\right\vert ^{\grave{p}_{k}}<\infty ,$ \ \ \ \ \
\ \ \ \ \ \ $\ \ \ \left( \frac{\lambda _{k}}{\lambda _{k}-\lambda _{k-1}}%
a_{nk}\right) _{k=0}^{\infty }\in c_{0}\left( q\right) $ $\left( \forall
n\in \mathbb{N}\right) $

\ \ \ \ \ \ \ \ \ \ \ \ \ \ \ \ \ \ \ \ $\uparrow \left( 4.18\right) $ \ \ \
\ \ \ \ \ \ \ \ \ \ \ \ \ \ \ \ \ \ \ \ \ \ \ \ \ \ \ \ \ \ \ \ \ \ \ \ \ $%
\uparrow \left( 4.19\right) $\ \

$\ \left( \frac{\lambda _{k}}{\lambda _{k}-\lambda _{k-1}}a_{nk}\right)
_{k=0}^{\infty }\in c\left( q\right) $ $\left( \forall n\in \mathbb{N}%
\right) $\ \ \ \ \ \ \ \ \ $\left( \frac{\lambda _{k}}{\lambda _{k}-\lambda
_{k-1}}a_{nk}\right) _{k=0}^{\infty }\in \ell _{\infty }\left( q\right) $ $%
\left( \forall n\in \mathbb{N}\right) $

\ \ \ \ \ \ \ \ \ \ \ \ \ \ \ \ \ \ \ $\uparrow \left( 4.20\right) $ \ \ \ \
\ \ \ \ \ \ \ \ \ \ \ \ \ \ \ \ \ \ \ \ \ \ \ \ \ \ \ \ \ \ \ \ \ \ \ \ \ $%
\uparrow \left( 4.21\right) $

By using $\left( 4.3\right) ,\left( 4.5\right) $ and Corollary 2, we have
the following results:

\begin{theorem}
We have

$\left( i\right) $ $A\in \left( \ell \left( \lambda ,p\right) :\ell \left(
q\right) \right) $ if and only if $\left( 4.6\right) ,\left( 4.7\right)
,\left( 4.8\right) $ and $\left( 4.19\right) $\ hold.

$\left( ii\right) $ $A\in \left( \ell \left( \lambda ,p\right) :c_{0}\left(
q\right) \right) $ if and only if $\left( 4.9\right) ,\left( 4.10\right)
,\left( 4.11\right) $ and $\left( 4.19\right) $\ hold.

$\left( iii\right) ${\small \ }$A\in \left( \ell \left( \lambda ,p\right)
:c\left( q\right) \right) ${\small \ }if and only if $\left( 4.12\right)
{\small ,}\left( 4.13\right) ,\left( 4.14\right) ,\left( 4.15\right) ,\left(
4.16\right) $ and $\left( 4.20\right) $\ hold.

$\left( iv\right) $ $A\in \left( \ell \left( \lambda ,p\right) :\ell
_{\infty }\left( q\right) \right) $ if and only if $\left( 4.17\right)
,\left( 4.18\right) $ and $\left( 4.21\right) $\ hold.
\end{theorem}

\end{document}